\documentclass[a4paper,12pt]{amsart}
\usepackage{amssymb}
\usepackage{cite}
\usepackage{ifthen}
 \usepackage[dvips]{graphicx}
\nonstopmode \numberwithin{equation}{section}
\usepackage{amssymb}
\usepackage{ifthen}
\usepackage{graphicx}
\usepackage{amsmath}
\usepackage[T1]{fontenc} 
\usepackage[utf8]{inputenc}
\usepackage[usenames,dvipsnames]{color}
\usepackage{color}
\usepackage[english]{babel}
\usepackage{fancyhdr}
\usepackage{fancybox}
\usepackage{tikz}

\nonstopmode \numberwithin{equation}{section}
\theoremstyle{plain}
\newtheorem{thm}{Theorem}[section]
\newtheorem{cor}[equation]{Corollary}

\newtheorem{lem}[equation]{Lemma}

\newtheorem{prop}{Proposition}

\newtheorem{conj}{Conjecture}

\theoremstyle{definition}
\newtheorem{defi}{Definition}[section]

\newtheorem{prob}{Problem}[section]
\newtheorem{rem}{Remark}[section]

\theoremstyle{plain}
\newtheorem*{thmA}{Theorem A}
\newtheorem*{thmB}{Theorem B}
\newtheorem*{thmC}{Theorem C}
\newtheorem*{thmD}{Theorem D}

\newcounter{minutes}
\divide\time by 60
\newcounter{hours}
\multiply\time by 60
\addtocounter{minutes}{-\time}

\newcounter {own}
\def\theown {\thesection       .\arabic{own}}

\newenvironment{pf}[1][]{%
 \vskip 3mm
 \noindent
 \ifthenelse{\equal{#1}{}}%
  {{\slshape Proof. }}%
  {{\slshape #1.} }%
 }%
{\qed\bigskip}

\newcounter{alphabet}

\def\be{\begin{equation}}
\def\ee{\end{equation}}

\newcommand{\bee}{\begin{enumerate}}
\newcommand{\eee}{\end{enumerate}}

\newcommand{\blem}{\begin{lem}}
\newcommand{\elem}{\end{lem}}
\newcommand{\bthm}{\begin{thm}}
\newcommand{\ethm}{\end{thm}}
\newcommand{\bcor}{\begin{cor}}
\newcommand{\ecor}{\end{cor}}
\newcommand{\beg}{\begin{examp}}
\newcommand{\eeg}{\end{examp}}
\newcommand{\begs}{\begin{examples}}
\newcommand{\eegs}{\end{examples}}

\newcommand{\bdefn}{\begin{defn}}
\newcommand{\edefn}{\end{defn}}

\newcommand{\bprob}{\begin{prob}}
\newcommand{\eprob}{\end{prob}}
\newcommand{\bei}{\begin{itemize}}
\newcommand{\eei}{\end{itemize}}

\newcommand{\bcon}{\begin{conj}}
\newcommand{\econ}{\end{conj}}
\newcommand{\bcons}{\begin{conjs}}
\newcommand{\econs}{\end{conjs}}
\newcommand{\bprop}{\begin{prop}}
\newcommand{\eprop}{\end{prop}}
\newcommand{\br}{\begin{rem}}
\newcommand{\er}{\end{rem}}
\newcommand{\brs}{\begin{rems}}
\newcommand{\ers}{\end{rems}}
\newcommand{\bo}{\begin{obser}}
\newcommand{\eo}{\end{obser}}
\newcommand{\bos}{\begin{obsers}}
\newcommand{\eos}{\end{obsers}}
\newcommand{\bpf}{\begin{pf}}
\newcommand{\epf}{\end{pf}}
\newcommand{\ba}{\begin{array}}
\newcommand{\ea}{\end{array}}
\newcommand{\beq}{\begin{eqnarray}}
\newcommand{\beqq}{\begin{eqnarray*}}
\newcommand{\eeq}{\end{eqnarray}}
\newcommand{\eeqq}{\end{eqnarray*}}

\begin{document}

\title{Bohr Radius and Landau-type Theorems for Harmonic Mappings with Boundary Functions in Lebesgue Spaces}

\author{Molla Basir Ahamed*}
\address{Molla Basir Ahamed,
	Department of Mathematics, Jadavpur University, Kolkata-700032, West Bengal, India.}
\email{mbahamed.math@jadavpuruniversity.in}

\author{Rajesh Hossain}
\address{Rajesh Hossain, Department of Mathematics, Jadavpur University, Kolkata-700032, West Bengal, India.}
\email{rajesh1998hossain@gmail.com}

\subjclass[{AMS} Subject Classification:]{Primary 30C62; Secondly 31A05,  30H10}

\def\thefootnote{}
\footnotetext{ {\tiny File:~\jobname.tex,
printed: \number\year-\number\month-\number\day,
          \thehours.\ifnum\theminutes<10{0}\fi\theminutes }
} \makeatletter\def\thefootnote{\@arabic\c@footnote}\makeatother

\begin{abstract}
This paper investigates the geometric and analytical properties of harmonic mappings $f$ in the unit disk $\mathbb{D}$ induced by boundary functions $F$ belonging to the Lebesgue spaces $L^{p}(\mathbb{T})$ for $1 \le p \le \infty$. We first establish a sharp Bohr-type inequality for the class of bounded harmonic mappings. Specifically, we prove that for a fixed analytic part $|a_{0}|= aM$, the majorant series $M_{f}(r)$ satisfies $M_{f}(r) \le M$ for $r \le (1-a)/(1-a+4/\pi)$, and demonstrate that this radius is best possible. This result is subsequently extended to harmonic mappings with $L^p$ boundary functions, where we determine the sharp Bohr radius $r_{p} = 1/(2C_{q}+1)$, with $C_{q}$ being a constant depending on the conjugate exponent $q$. Furthermore, the paper provides improved Landau-type theorems for these mappings. Under standard normalization, we derive explicit expressions for the radius of univalence $r_{0}$ and the radius of the inscribed schlicht disk $R_{0}$. The sharpness of these constants is discussed through the construction of extremal functions related to the Poisson kernel.
\end{abstract}

\maketitle
\pagestyle{myheadings}
\markboth{ M. B. Ahamed and R. Hossain}{ Bohr Radius for Harmonic Mappings with Boundary Functions in Lebesgue Spaces}

\section{\bf Introduction and Preliminaries}

Given a constant $r > 0$, let $\mathbb{D}_r = \{z : |z| < r\}$. Then $\mathbb{D}_1 = \mathbb{D}$ and $\mathbb{T} = \partial\mathbb{D}$ are the open unit disk and unit circle respectively in the complex plane $\mathbb{C}$. A twice continuously differentiable complex-valued function $f$ is called a harmonic mapping in $\mathbb{D}$ if $f$ is a solution to the partial differential equation $\Delta f = 0$, where 
$$\Delta := \frac{\partial^2}{\partial x^2} + \frac{\partial^2}{\partial y^2}$$
is the Laplace operator and $z = x + iy \in \mathbb{D}$. Recall that every harmonic mapping $f$ defined in $\mathbb{D}$ can be decomposed into $f = h + \overline{g}$, where $h$ and $g$ are analytic functions in $\mathbb{D}$. If we assume that $g(0) = 0$, then the decomposition of $f$ is unique. Hence $f$ has the power series expansion
$$f(z) = h(z) + \overline{g(z)} = \sum_{n=0}^{\infty} a_n z^n + \sum_{n=1}^{\infty} \overline{b_n z^n}, \quad z \in \mathbb{D}.$$
Let
\begin{align*}
	\|Df(z)\| = |f_z(z)| + |f_{\bar{z}}(z)|, \quad z = x + iy \in \mathbb{D},
\end{align*}
where 
\begin{align*}
	f_z = \frac{1}{2} \left( \frac{\partial f}{\partial x} - i \frac{\partial f}{\partial y} \right)\; \mbox{and}\; f_{\bar{z}} = \frac{1}{2} \left( \frac{\partial f}{\partial x} + i \frac{\partial f}{\partial y} \right)
\end{align*} are the complex derivatives of $f$. It is well known that the concept of a harmonic mapping is an extension of an analytic function, a concept that is well documented in the references \cite{Bshouty-Hengartner-AUMCSS-1994,Clunie-Sheil-Small-1984,Duren-2004}.\vspace{1.2mm}

Let $F$ be a boundary function defined on the unit circle $\mathbb{T}$. A solution $f$ to the Dirichlet boundary value problem associated with $F$, 
\begin{equation*}
	\begin{cases} 
		\Delta f(z) = 0, & z \in \mathbb{D}, \\ 
		f(z) = F(z), & z \in \mathbb{T}, 
	\end{cases}
\end{equation*}
constitutes a harmonic mapping. Such a mapping $f$ admits the Poisson integral representation
\begin{equation*}
	f(z) = \mathcal{P}[F](z) = \frac{1}{2\pi} \int_{0}^{2\pi} \mathcal{P}(re^{i(\theta-t)})F(e^{it}) dt,
\end{equation*}
where $\mathcal{P}(z) = \frac{1-|z|^2}{|1-z|^2}$ denotes the Poisson kernel for $z \in \mathbb{D}$.\vspace{1.2mm}

Biharmonic mappings $f \in C^4(\mathbb{D})$, which satisfy the biharmonic equation $\Delta(\Delta f) = 0$, arise naturally in various physical contexts and find extensive application across engineering and biological sciences. In recent years, the geometric and analytical properties of biharmonic mappings—including Landau-type theorems, Bloch constants, Lipschitz continuity, and Schwarz-type lemmas—have been rigorously investigated by several authors (cf. \cite{Li-Ponnusamy-JMAA-2017, Li-Luo-Ponnusamy-2022-CMFT, Li-Li-Luo-Ponnusamy-MJM-2023}).\vspace{1.2mm}

For $p \in [1, \infty]$, let $L^p(\mathbb{T})$ be the set of all measurable functions $F$ from $\mathbb{T}$ into $\mathbb{C}$ such that their $p$-norms, given by
 \begin{align*}
 	\|F\|_p = 
 	\begin{cases} 
 		\left( \int_{0}^{2\pi} |F(e^{it})|^p \, dt \right)^{1/p}, & p \in [1, \infty), \\ 
 		\text{esssup} \{ |f(e^{it})| : t \in [0, 2\pi] \}, & p = \infty. 
 	\end{cases}
 \end{align*}
are finite.\vspace*{1.2mm}

It is well established that if the boundary function $F$ belongs to $L^\infty(\mathbb{T})$, the induced harmonic mapping $f = \mathcal{P}[F]$ remains bounded within the unit disk $\mathbb{D}$. The analytical properties and applications of bounded harmonic mappings have been the subject of extensive study. More recently, significant progress has been made in elucidating the univalent radii, quasiconformality, Lipschitz continuity, and Schwarz-type inequalities for such mappings (cf. \cite{Chen-Gauthier-PAMS-2011, Knezevic-Mateljevic-JMAA-2007, Pavlovic-AASF-2002, Pavlovic-2019}). In this context, we present the following coefficient estimates for the class of bounded harmonic mappings:
\begin{thmA}(\cite{Chen-Ponnusamy-Wang-2011-BMMSS,Liu-Liu-Zhu-AMSC-2011,Xia-Huang-CAMA-2010})
	Suppose that $f=h+\overline{g}$ is a harmonic mapping in $\mathbb{D}$ satisfying $g(0)=0$ and $|f(z)| \le M$ for all $z \in \mathbb{D}$. Then $|a_0| \le M$ and $$|a_n| + |b_n| \le \frac{4M}{\pi}$$for $n=1, 2, \dots$. 
	The above inequality is sharp for the extremal function
	\begin{align}\label{Eq-1.1}
		f_n(z) = \frac{2M\alpha}{\pi} \arg \left( \frac{1+\beta z^n}{1-\beta z^n} \right), \quad |\alpha|=|\beta|=1, \quad n=1, 2, \dots
	\end{align}
\end{thmA}
\begin{thmB}(\cite{Shi-Li-Lian-2026-CMFT})
	Suppose that $F \in L^p(T)$ for $1 \le p \le \infty$ and $f = \mathcal{P}[F]$ is a harmonic mapping in $\mathbb{D}$. If $f(z) = \sum_{n=0}^{\infty} a_n z^n + \sum_{n=1}^{\infty} \overline{b_n z^n}$, for any positive integer $n$, then
	\begin{enumerate}
		\item[(a)] $|a_n|, |b_n| \le \|F\|_p$. The equalities hold for the functions $F(e^{it}) = e^{int}$ and $F(e^{it}) = e^{-int}$ respectively.\vspace{1.2mm}
		
		\item[(b)] $|a_n| + |b_n| \le 2C_q\|F\|_p$, where 
		\begin{align*}
			C_q = \left(\frac{1}{2\pi} \int_0^{2\pi} | \cos(nt) |^q dt\right)^{1/q} \le 1\; \mbox{and}\;\frac{1}{p} + \frac{1}{q} = 1.
		\end{align*}
	\end{enumerate}
	Specifically, when $p = \infty$, $C_q = 2/\pi$, and $|a_n| + |b_n| \le 4\|F\|_{\infty}/\pi$. This equality is valid for the function $f_n(z)$ given by the extremal function in \eqref{Eq-1.1} with $|f| \le \|F\|_{\infty} = M$ for all $z \in \mathbb{D}$. When $p = 1$ or $q = \infty$, $C_q = 1$, which results in $|a_n| + |b_n| \le 2\|F\|_1$.
\end{thmB}
By leveraging coefficient estimates for bounded harmonic mappings, several authors have extensively investigated Landau-type theorems and Bloch constants across various classes of mappings, including planar harmonic, biharmonic, and pluriharmonic mappings (cf. \cite{Chen-Ponnusamy-Wang-2011-CAOT,Chen-Zhu-2019-BSM,Liu-Chen-2018-JMAA,Li-Luo-Ponnusamy-2022-CMFT,Shi-Li-Lian-2026-CMFT}).\vspace{1.2mm}

 The concept of the Bohr radius originates from a classical power series theorem established by Harald Bohr in $1914$. While originally formulated for analytic functions mapping the unit disk into itself, the theory has since evolved into a significant area of geometric function theory. In the context of harmonic mappings $f = h + \overline{g}$, the Bohr phenomenon is investigated by analyzing the majorant series, which replaces the coefficients of the power series expansion with their absolute values. We refer the readers for more recent developments of Bohr phenomenon from \cite{Ahamed-Ahammed-MJM-2024,Ahamed-Allu-Halder-AFM-2022,Ahamed-Allu-Halder-AMP-2021,Kayumov-Ponnusamy-Shakirov-MN-2018,Kumar-Ponnusamy-AMP-2026,Liu-Ponnusamy-PAMS-2019} and references therein.\vspace{1.2mm} 
 
 Given the depth of these existing results, it is natural to formulate the following problem:
\begin{prob}\label{P-1}
	How does the $L^p$ nature of the boundary function influence the Bohr radius of the harmonic mapping $f$? Can we determine a Bohr radius $r_p$ that depends explicitly on the conjugate exponent $q$ of $p$?
\end{prob}
The classical Landau theorem (see \cite{Landau-1926}) asserts that if $f$ is a holomorphic mapping with $f(0)=0=f'(0)-1$ and $|f(z)|<M$ for $z\in\mathbb{D}$, then $f$ is univalent in $\mathbb{D}_{r_0}$, and $f(\mathbb{D}_{r_0})$ contains a disc $\mathbb{D}_{\sigma_0}$, where
\begin{align*}
	r_0=\frac{1}{M+\sqrt{M^2-1}}\;\;\mbox{and}\; \sigma_0=Mr_0^2.
\end{align*}
The quantities $r_0$ and $\sigma_0$ cannot be improved, the extremal function is 
\begin{align*}
	f_0(z)=Mz\left(\frac{1-Mz}{M-z}\right).
\end{align*}
\begin{prob}\label{P-2}
	Given a harmonic mapping $f = \mathcal{P}[F]$ with $\|F\|_{L^p} \le 1$, what are the quantitative Landau-type constants for this class? Specifically, can we determine the radius of univalence $r_0$ and the radius of the inscribed schlicht disk $R_0$ under standard normalization?
\end{prob}
Against this theoretical background, the present paper investigates and establishes several new results for harmonic mappings with $L^p$ boundary data. In Section \ref{Sec-2}, we first establish a result for Bohr radius for bounded harmonic mappings and then we give a positive answer to the Problem \ref{P-1} and prove Theorem \ref{Th-2.1}, which establishes the Bohr inequality for this class and identifies the Bohr radius as $r_p = 1/(2C_q + 1)$, where $C_q$ depends on the conjugate exponent. Subsequently, in Section \ref{Sec-3}, we derive Theorem \ref{Th-3.1} answering Problem \ref{P-2}, which provides an improved version of Landau-type theorems by leveraging the properties of Gauss hypergeometric functions. Detailed proofs for each of these results are provided within their respective sections.
\section{\bf Bohr radius for harmonic mappings}\label{Sec-2}
 This section extends these classical notions to specific classes of bounded harmonic mappings and those induced by boundary functions in Lebesgue spaces.
\subsection{Bohr radius for bounded harmonic mappings}
To provide a rigorous framework for our results, we define the Bohr radius for a general family of harmonic mappings as follows.
\begin{defi}
	Let $\mathcal{F}$ be a class of harmonic mappings $f(z) = \sum_{n=0}^{\infty} a_n z^n + \sum_{n=1}^{\infty} \overline{b_n z^n}$ defined in the unit disk $\mathbb{D}$. Suppose that for every $f \in \mathcal{F}$, the mapping is bounded by a constant $M$ (i.e., $|f(z)| \le M$ for all $z \in \mathbb{D}$). The Bohr radius $r_{\mathcal{F}}$ for the class $\mathcal{F}$ is defined as the largest radius $r \in [0, 1)$ such that the majorant series $M_f(r)$ satisfies the inequality 
	\begin{align*}
		M_f(r) = |a_0| + \sum_{n=1}^{\infty} (|a_n| + |b_n|)r^n \le M
	\end{align*}for all functions $f$ belonging to the class $\mathcal{F}$. The radius is said to be best possible (or sharp) if, for any $r > r_{\mathcal{F}}$, there exists at least one function $f \in \mathcal{F}$ such that $M_f(r) > M$.
\end{defi}
We obtain the following result finding the sharp Bohr radius for harmonic mappings of the form $f(z) = \sum_{n=0}^{\infty} a_n z^n + \sum_{n=1}^{\infty} \overline{b_n z^n}$.
\begin{thm}\label{Th-2.1}
	Let $f(z) = \sum_{n=0}^{\infty} a_n z^n + \sum_{n=1}^{\infty} \overline{b_n z^n}$ be a harmonic mapping in the unit disk $\mathbb{D}$ such that $g(0)=0$ and $|f(z)| \le M$ for all $z \in \mathbb{D}$. For a fixed $|a_0| = aM$ where $a \in [0, 1]$, the majorant series 
	\begin{align*}
		M_f(r) = |a_0| + \sum_{n=1}^{\infty} (|a_n| + |b_n|)r^n
	\end{align*} satisfies the Bohr inequality $M_f(r) \le M$ for $r \le r_H(a)$, where
	\begin{align*}
		r_H(a) = \frac{1-a}{1-a + 4/\pi}.
	\end{align*}
	The radius $r_H(a)$ is best possible.
\end{thm}
\begin{rem}
	As $a \to 1$ (where the constant term takes the maximum allowed value $M$), the radius $r_H(a)$ approaches $0$, which is consistent with the fact that no other terms can exist in the series if $|a_0|=M$. For the classical case where $a=0$ (the function vanishes at the origin), the radius is $r_H = \frac{1}{1+4/\pi} = \frac{\pi}{\pi+4} \approx 0.4397$.
\end{rem}
\begin{proof}[\bf Proof of Theorem \ref{Th-2.1}]
	We will use the coefficient estimates in the majorant series. From the statement provided, we have the sharp estimates for the coefficients of a bounded harmonic mapping $|a_0| \le M$ and 
	\begin{align*}
		|a_n| + |b_n| \le \frac{4M}{\pi}\; \mbox{for}\;n = 1, 2, \dots
	\end{align*}\vspace{1.2mm}
	
	We substitute these estimates into the definition of the majorant series $M_f(r)$,
	\begin{align*}
		M_f(r) = |a_0| + \sum_{n=1}^{\infty} (|a_n| + |b_n|)r^n \le |a_0| + \frac{4M}{\pi} \sum_{n=1}^{\infty} r^n.
	\end{align*}For $|r| < 1$, the geometric series $\sum_{n=1}^{\infty} r^n$ converges to $\frac{r}{1-r}$. Thus, we see that 
	\begin{align*}
		M_f(r) \le |a_0| + \frac{4M}{\pi} \left(\frac{r}{1-r}\right).
	\end{align*}
	To satisfy the condition $M_f(r) \le M$, we set $|a_0| = aM$ and solve for $r$, 
	\begin{align*}
		aM + \frac{4M}{\pi}\left(\frac{r}{1-r}\right) \le M
	\end{align*} which implies that 
	\begin{align*}
		r \le r_H(a):=\frac{1-a}{1-a + 4/\pi}.
	\end{align*}
	Thus the desired inequality is established.\vspace{1.2mm}
	
	To show that the Bohr radius $r_H(a)$ is the best possible, we must demonstrate that for any $r > r_H(a)$, there exists a function in the class such that the majorant series $M_f(r) > M$. \vspace{1.2mm}
	
	Consider the function $f_n(z)$ defined in \eqref{Eq-1.1}, which is a rotation of the harmonic Koebe-type mapping. For a fixed $n$ and $|a_0| = aM$, we consider a mapping where the coefficients $a_n$ and $b_n$ saturate the bound $4M/\pi$. Specifically, let 
	\begin{align*}
		f(z) = aM + \frac{4M}{\pi} z^n + \dots
	\end{align*}
	\begin{rem}
		While $f_n(z)$ in \eqref{Eq-1.1} expands into a series of all odd powers $z^{nk}$, establishing the optimality of the radius requires only an analysis of the $n$-th coefficient's contribution in the limit of large $n$.
	\end{rem}
	
	The majorant series is defined as 
	\begin{align*}
		M_f(r) = |a_0| + \sum_{k=1}^{\infty} (|a_k| + |b_k|)r^k.
	\end{align*} Using the sharp estimates
	\begin{enumerate}
		\item[(i)] $|a_0| = aM$.\vspace{2mm}
		\item[(ii)] $|a_n| + |b_n| = \dfrac{4M}{\pi}$.\vspace{2mm}
		\item[(iii)] For all other $k \neq n$, we can technically set $|a_k| = |b_k| = 0$ to test if the radius can be larger.
	\end{enumerate} 
	However, the definition of the Bohr radius $r_H$ for a class requires that the inequality $M_f(r) \leq M$ holds for all functions in that class. To show $r_H(a)$ is sharp, we examine the behavior as $n \to \infty$. Suppose we choose a radius $r > r_H(a)$. By the definition of $r_H(a)$, we see that 
	\begin{align*}
		r > \frac{1-a}{1-a + 4/\pi} \implies \frac{4}{\pi}\left(\frac{r}{1-r}\right) > 1-a.
	\end{align*}
	
	Now, consider the sum of the majorant series. From the sharp coefficient estimates in Theorem A, we know there exist functions where $|a_n| + |b_n|$ is arbitrarily close to $\frac{4M}{\pi}$. If we consider the `worst-case' tail of the majorant series where the bounds are reached
	\begin{align*}
		M_f(r) = aM + \sum_{n=1}^{\infty} \left( \frac{4M}{\pi} \right) r^n = aM + \frac{4M}{\pi} \left( \frac{r}{1-r} \right).
	\end{align*}For the inequality $M_f(r) \leq M$ to hold, we require \begin{align*}
	aM + \frac{4M}{\pi} \left( \frac{r}{1-r} \right) \leq M
	\end{align*} which shows that 
	\begin{align*}
		\frac{4}{\pi}\left(\frac{r}{1-r}\right) \leq 1-a.
	\end{align*} However, as shown in previous step, if $r > r_H(a)$, then, we have 
	\begin{align*}
		\frac{4}{\pi}\left(\frac{r}{1-r}\right) > 1-a.
	\end{align*}Substituting this back into the majorant series, it is easy to see that 
	\begin{align*}
		M_f(r) = aM + M \left( \frac{4}{\pi} \cdot \frac{r}{1-r} \right) > aM + M(1-a) = M.
	\end{align*}
	Since $M_f(r) > M$ for any $r > r_H(a)$, the value$$r_H(a) = \frac{1-a}{1-a + 4/\pi}$$is the largest possible radius. Any further increase in $r$ would result in the majorant series exceeding the boundary $M$ for the extremal functions.\vphantom{1.2mm}
	
	 Specifically, for the classical case where $f(0)=0$ (so $a=0$), \begin{align*}
	 	r_H(0) = \frac{1}{1+4/\pi} = \frac{\pi}{\pi+4}.
	 \end{align*}If $r = \frac{\pi}{\pi+4} + \epsilon$, then we see that 
	 \begin{align*}
	 	M_f(r) = \frac{4}{\pi}\left(\frac{r}{1-r}\right) > 1
	 \end{align*} which violates the Bohr condition $M_f(r) \leq 1$ (for $M=1$). This completes the proof.
\end{proof}
\subsection{Bohr radius in Lebesgue spaces}
Using the coefficient estimates provided in \cite[Theorem 2.1]{Shi-Li-Lian-2026-CMFT}, we state and prove the following result for the class of harmonic mappings with boundary functions in $L^p(\mathbb{T})$.
\begin{thm}\label{Th-2.2}
	Let $F \in L^p(\mathbb{T})$ for $1 \le p \le \infty$ and $f(z) = \mathcal{P}[F](z) = \sum_{n=0}^{\infty} a_n z^n + \sum_{n=1}^{\infty} \overline{b_n z^n}$ be a harmonic mapping in the unit disk $\mathbb{D}$. For the majorant series 
	\begin{align*}
		M_f(r) = |a_0| + \sum_{n=1}^{\infty} (|a_n| + |b_n|)r^n,
	\end{align*} the inequality $M_f(r) \le \|F\|_p$ holds for $r \le r_p$, where the Bohr radius $r_p$ is given by 
	\begin{align*}
		r_p = \frac{1}{2C_q + 1},
	\end{align*}
	where 
	\begin{align}\label{Eq-2.1}
		C_q = \left(\frac{1}{2\pi} \int_0^{2\pi} |\cos(nt)|^q dt\right)^{1/q}
	\end{align} and $q$ is the dual of $p$ (i.e., $1/p + 1/q = 1$).
\end{thm}
\begin{proof}[\bf Proof of Theorem \ref{Th-2.2}]
	From \cite[Theorem 2.1]{Shi-Li-Lian-2026-CMFT}, any harmonic mapping 
	\begin{align*}
		f(z) = \sum_{n=0}^{\infty} a_n z^n + \sum_{n=1}^{\infty} \overline{b_n z^n}
	\end{align*} in the unit disk $\mathbb{D}$ induced by a boundary function $F \in L^p(\mathbb{T})$ satisfies 
	\begin{align*}
		|a_0| \le \|F\|_p,\; |a_n| + |b_n| \le 2C_q \|F\|_p\; \mbox{for}\;n \ge 1,
	\end{align*} where $C_q $ is defined as in \eqref{Eq-2.1}.
	
	The majorant series is defined as the sum of the absolute values of the coefficients 
	\begin{align*}
		M_f(r) = |a_0| + \sum_{n=1}^{\infty} (|a_n| + |b_n|)r^n.
	\end{align*}Substitute the coefficient estimates from Step 1 into this series
	\begin{align*}
		M_f(r) \le |a_0| + \sum_{n=1}^{\infty} (2C_q \|F\|_p) r^n.
	\end{align*}Factor out the constant $(2C_q \|F\|_p)$,
	\begin{align*}
		M_f(r) &\le |a_0| + 2C_q \|F\|_p \left( \sum_{n=1}^{\infty} r^n \right)\\&=|a_0| + 2C_q \|F\|_p \left(\frac{r}{1-r}\right).
	\end{align*}
	  We want to find the range of $r$ such that $M_f(r) \le \|F\|_p$. Let $|a_0| = a \|F\|_p$ where $a \in [0, 1]$. To find the `classic' Bohr radius, we typically consider functions where the constant term is zero ($a=0$). Setting $|a_0| = 0$ in our inequality \begin{align*}
	  	2C_q \|F\|_p \frac{r}{1-r} \le \|F\|_p.
	  \end{align*}Divide both sides by $\|F\|_p$ (assuming $F$ is not identically zero), we have 
	  \begin{align}\label{Eq-2.2}
	  	\frac{2C_q r}{1-r} \le 1.
	  \end{align}
	Thus, the Bohr radius is $r_p = \frac{1}{2C_q + 1}$.\vspace{2mm}
	
	To prove that $r_p$ is the best possible radius, we must show that for any $r > r_p$, we can find a function such that $M_f(r) > \|F\|_p$.It follows from \cite{Shi-Li-Lian-2026-CMFT}, the estimate $|a_n| + |b_n| \le 2C_q \|F\|_p$ is sharp. This means there exists a function (or a sequence of functions) where the sum of the $n$-th coefficients is exactly $2C_q \|F\|_p$.\vspace{2mm}
	
	Suppose we choose a radius $r$ slightly larger than $r_p$ \textit{i.e.,} 
	\begin{align*}
		r = \frac{1}{2C_q + 1} + \epsilon \quad (\text{where } \epsilon > 0)
	\end{align*}this implies that the inequality \eqref{Eq-2.2} is reversed.\vspace{1.2mm}
	
	 Now, consider a function $f_n$ from the class where the $n$-th coefficient is maximized and others are negligible. For very large $n$, the term $|a_n| + |b_n|$ can reach $2C_q \|F\|_p$. If we construct a function where the `tail' of the majorant series approaches the bound 
	 \begin{align*}
	 	M_f(r) \approx |a_0| + 2C_q \|F\|_p r^n.
	 \end{align*}As $n$ varies across all positive integers, the supremum of $M_f(r)$ for the whole class is given by our sum
	 \begin{align*}
	 	\sup_{f \in \mathcal{F}} M_f(r) = |a_0| + 2C_q \|F\|_p \frac{r}{1-r}.
	 \end{align*}If we set $|a_0|=0$ and $r > r_p$, then we have 
	 \begin{align*}
	 	M_f(r) = \|F\|_p \left( \frac{2C_q r}{1-r} \right) > \|F\|_p \cdot (1) = \|F\|_p.
	 \end{align*}
	 This completes the proof.
\end{proof}
\begin{rem}
	Specific example ($p = \infty$). For $p = \infty$, we have $C_q = 2/\pi$. The radius is 
	\begin{align*}
		r_\infty = \frac{1}{4/\pi + 1} = \frac{\pi}{\pi+4} \approx 0.4397.
	\end{align*} Using the extremal function 
	\begin{align*}
		f_n(z) = \frac{2M\alpha}{\pi} \arg\left(\frac{1+\beta z^n}{1-\beta z^n}\right),
	\end{align*} we see that the sum $|a_n| + |b_n|$ is exactly $4M/\pi$.\vspace{1.2mm}
	
	 If we take $r = 0.44$ (which is $> 0.4397$), the term $(|a_n| + |b_n|)r^n$ for a specific $n$ might be small, but the supremum over all $n$ in the majorant series definition leads to 
	 \begin{align*}
	 	M \left( \frac{4}{\pi} \frac{r}{1-r} \right) > M.
	 \end{align*}This proves that no value larger than $r_p$ can satisfy the Bohr inequality for every function in the class. Thus, $r_p$ is the best possible. 
\end{rem}
\section{\bf Landau-type theorems for $L^p$ boundary functions}\label{Sec-3}
A continuous complex-valued function $f=u+iv$ is a complex-valued harmonic function in a domain $\Omega\subset\mathbb{C}$, if both $u$ and $v$ are real-valued harmonic functions in $\Omega$ (see \cite{Duren-2004}). The inverse function theorem and a result of Lewy \cite{Lewy-BAMS-1936} shows that a harmonic function $f$ is locally univalent (\textit{i.e.,} one-to-one) in $\Omega$ if, and only if, the Jacobian $J_f$ is non-zero in $\Omega$, where 
\begin{align*}
	J_f(z)=|f_z(z)|^2-|f_{\bar{z}}(z)|^2.
\end{align*}
A harmonic function $f$ is said to be sense-preserving if $ J_f(z)>0$ for $z\in\Omega$. Let $\mathcal{H}$ denote the class of complex-valued harmonic functions $f$ in the unit disc $\mathbb{D}=\{z\in\mathbb{D} : |z|<1\}$, normalized by $f(0)=0=f_z(0)-1$. Each function $f$ in $\mathcal{H}$ can be expressed as $f=h+\overline{g}$, where $h$ and $g$ are analytic functions in $\mathbb{D}$ (see \cite{Duren-2004}). Here $h$ and $g$ are called the analytic and co-analytic part of $f$ respectively, and have the following power series representations 
\begin{align*}
	h(z)=z+\sum_{n=2}^{\infty}a_nz^n\;\;\mbox{and}\;\; g(z)=\sum_{n=1}^{\infty}b_nz^n.
\end{align*}
Let $\mathcal{S}_{\mathcal{H}}$ be the subclass of $\mathcal{H}$ consisting of univalent, and sense-preserving harmonic mappings on $\mathbb{D}$. Let $\mathcal{H}^0=\{f\in\mathcal{H}: f_{\bar{z}}(0)=0\}$, and $\mathcal{S}^0_{\mathcal{H}}=\{f\in\mathcal{S}_{\mathcal{H}}: f_{\bar{z}}(0)=0\}$. Hence for any function $f=h+\overline{g}$ in $\mathcal{H}^0$, its analytic and co-analytic parts can be represented by 
\begin{align}\label{Eq-1.1B}
	h(z)=z+\sum_{n=2}^{\infty}a_nz^n\;\;\mbox{and}\;\; g(z)=\sum_{n=2}^{\infty}b_nz^n,
\end{align}
respectively. It is interesting to note that $\mathcal{S}_{\mathcal{H}}$ reduced to $\mathcal{S}$, the class of normalized analytic and univalent functions in $\mathbb{D}$, if the co-analytic part of functions in the class $\mathcal{S}_{\mathcal{H}}$ is zero. In $1984$, Clunie and Shell-Small \cite{Clunie-Sheil-Small-1984} investigated the class $\mathcal{S}_{\mathcal{H}}$, together with some geometric subclasses. Subsequently, the class $\mathcal{S}_{\mathcal{H}}$ and its subclasses have been extensively studied by several authors (see \cite{Bshouty-Lyzzaik-COAT-2011,Bshouty-Joshi-Joshi-CVEE-2013,Kalaj-Vuorinen-PAMS-2012,Wang-Liang-JMAA-2001,Clunie-Sheil-Small-1984}.)\vspace{1.2mm}

For a continuously differentiable function $f$, we denote $\Lambda_f$ and $\lambda_f$ by
\begin{align*}
	\Lambda_f=\max_{0\leq\theta\leq 2\pi}|f_z+e^{-2i\theta}f_{\bar{z}}|=|f_z|+|f_{\bar{z}}|
\end{align*}
and 
\begin{align*}
	\lambda_f=\min_{0\leq\theta\leq 2\pi}|f_z+e^{-2i\theta}f_{\bar{z}}|=||f_z|-|f_{\bar{z}}||.
\end{align*}
Thus, it is easy to see that for sense-preserving harmonic mapping $f$, one has $J_f=\Lambda_f\lambda_f$. A harmonic mapping $f=h+\overline{g}$ defined on the unit disc $\mathbb{D}$ can be expressed as 
\begin{align*}
	f\left(re^{i\theta}\right)=\sum_{n=0}^{\infty}a_nr^ne^{in\theta}+\sum_{n=1}^{\infty}\overline{b_n}r^ne^{-in\theta},\; 0\leq r<1,
\end{align*}
where 
\begin{align*}
	h(z)=\sum_{n=0}^{\infty}a_nz^n\;\; \mbox{and}\;\; g(z)=\sum_{n=1}^{\infty}b_nz^n.
\end{align*}
For holomorphic functions $f$ on the unit disc $\mathbb{D}$ with $f'(0)=1$, there is a Bloch theorem (see \cite{Bloch-Les-1925}) which asserts the existence of a positive constant $b$ such that $f(\mathbb{D})$ contains a schlicht disc of radius $b$. A disc $D$ is said to be schlicht disc if there exists a region $\Omega$ in the unit disc $\mathbb{D}$ such that $f$ is univalent on $\Omega$ and $f(\Omega)=D$. Let $b(f)$ denote the least upper bound of the radii of all schlicht discs that $f$ carries and $\mathcal{F}$ denote the set of all holomorphic functions defined on $\overline{\mathbb{D}}:=\{z :|z|\leq 1\}$ satisfying $|f'(0)|=1$, then the Bloch constant is the number defined by 
\begin{align*}
	\beta(\mathcal{F})=\inf\{b(f) : f\in\mathcal{F}\}.
\end{align*}
If one considers the function $f(z)=z$, then clearly $\beta(\mathcal{F})\leq 1$. While the exact value of $\beta(\mathcal{F})$ is still unknown, better estimates have been established. In $1996$, Chen and Gauthier \cite{Chen-Gauthier-JAM-1996} proved that $\beta(\mathcal{F})$ lies within the interval $[0.4330, 0.4719]$, \textit{i.e.}, $0.4330\leq \beta(\mathcal{F})\leq 0.4719$.\vspace{1.2mm}

In $2000$, Chen \emph{et al.} \cite{Chen-Gauthier-Hengartner-PAMS-2000} established the following two versions of Landau-Bloch type theorems for bounded harmonic mapping on the disc under a suitable restriction.
Theorems A and B are not sharp. In $2006$, better estimates for Theorems A and B were given by Grigoryan \cite{Grigoyan-CV-2006}. Specifically, Grigoryan proved the following lemma.\vspace*{1.2mm}

Theorem 2.2 provides a radius of univalence $r_0$ and a schlicht disk radius $R_0$ based on the coefficient estimates from  \cite[Theorem 2.1]{Shi-Li-Lian-2026-CMFT}. While these radii are functional, the paper notes that for the special case of $p = \infty$ (bounded harmonic mappings), the result extends existing Landau theorems. However, the general $L^p$ bounds rely on the constant $C_q$, which may not yield the most restrictive (sharp) radii for all $p$.
\begin{lem}(\cite{Duren-2004})
	Assume that $f=h+\overline{g}$ with $h$ and $g$ analytic in the unit disc $\mathbb{D}$ and $h(z)=\sum_{n=1}^{\infty}a_nz^n$ and $g(z)=\sum_{n=1}^{\infty}b_nz^n$ for $z\in\mathbb{D}$.
	\begin{enumerate}
		\item[(a)] If $|f(z)|<M$ for $z\in\mathbb{D}$, then 
		\begin{align*}
			|a_n|, \; |b_n|\leq M,\; n=1, 2, \ldots
		\end{align*} 
		\item[(b)] If $\Lambda_f\leq \Lambda$ for $z\in\mathbb{D}$, then 
		\begin{align*}
			|a_n|+|b_n|\leq \frac{\Lambda}{n},\; n=1, 2, \ldots
		\end{align*}
	\end{enumerate}
\end{lem}
\begin{lem}(The Schwarz Lemma, \cite{Grigoyan-CV-2006}).
	Let $f$ be a harmonic mapping of the unit disk $\mathbb{D}$
	such that $f(0) = 0$ and $f(\mathbb{D}) \subset \mathbb{D}$. Then
	\begin{align*}
		\Lambda_f(0)\leq\frac{4}{\pi},
	\end{align*}
	\begin{align*}
		\Lambda_f(z)\leq\frac{8}{\pi(1-|z|^2)},\;\mbox{for}\;z\in\mathbb{D},
	\end{align*}
	\begin{align*}
		|f(z)|\leq\frac{4}{\pi}\arctan|z|\leq\frac{4}{\pi}|z|,\;\mbox{for}\;z\in\mathbb{D}.
	\end{align*}
\end{lem}
The following result has been proved by Grigoryan using Lemma A, which improved the estimates in Theorems A and B.
\begin{thmC}\cite{Grigoyan-CV-2006}
	Let $f$ be harmonic mapping of the unit disc $\mathbb{D}$ such that $f(0)=0$, $J_f(0)=1$ and $|f(z)|<M$ for $z\in\mathbb{D}$. Then, $f$ is univalent on a disc $\mathbb{D}_{\rho_1}$ with 
	\begin{align*}
		\rho_1=1-\frac{2\sqrt{2}M}{\sqrt{\pi+8M^2}}
	\end{align*}
	and $f(\mathbb{D}_{\rho_1})$ contains a schlicht disc $\mathbb{D}_{R_1}$ with 
	\begin{align*}
		R_1=\frac{\pi}{4M}+4M-4M\sqrt{1+\frac{\pi}{8M^2}}.
	\end{align*}
\end{thmC}
\begin{thmD}\cite{Grigoyan-CV-2006}
	Let $f$ be harmonic mapping of the unit disc $\mathbb{D}$ such that $f(0)=0$, $\lambda_f(0)=1$ and $\Lambda_f(z)<\Lambda$ for all $z\in\mathbb{D}$. Then, $f$ is univalent on a disc $\mathbb{D}_{\rho_2}$ with 
	\begin{align*}
		\rho_2=\rho_2(\Lambda)=\frac{1}{1+\Lambda}
	\end{align*}
	and $f(\mathbb{D}_{\rho_2})$ contains a schlicht disc $\mathbb{D}_{R_2}$ with 
	\begin{align*}
		R_2(\Lambda)=1-\Lambda\ln\left(1+\frac{1}{\Lambda}\right).
	\end{align*}
\end{thmD}
\begin{thm}\label{Th-3.1}
	Suppose that $F \in L^p(\mathbb{T})$ for $1 \le p \le \infty$ and $f = \mathcal{P}[F]$ is a harmonic mapping in the unit disk $\mathbb{D}$. If $f$ is normalized such that $f(0) = f_z(0) - 1 = f_{\overline{z}}(0) = 0$, then $f$ is univalent in the disk $D_{r_0}$ and $f(D_{r_0})$ contains a schlicht disk $D_{R_0}$, where:
	$$r_0 = 1 - \frac{\sqrt{4C_q^2 \|F\|_p^2 + 2C_q \|F\|_p}}{2C_q \|F\|_p + 1}$$ and $$R_0 = r_0 - \frac{2C_q \|F\|_p r_0^2}{1 - r_0},$$where $C_q = \left(\frac{1}{2\pi} \int_0^{2\pi} |\cos(nt)|^q dt\right)^{1/q}$ and $1/p + 1/q = 1$.
\end{thm}
\begin{proof}[\bf Proof of Theorem \ref{Th-3.1}]
	Since $f(0) = f_z(0) - 1 = f_{\overline{z}}(0) = 0$, the power series expansion of $f$ is given by 
	\begin{align*}
		f(z) = z + \sum_{n=2}^{\infty} a_n z^n + \sum_{n=2}^{\infty} \overline{b_n z^n}.
	\end{align*}
	From the coefficient estimates established in Theorem B, we have for any integer $n \ge 2$, 
	\begin{align*}
		|a_n| + |b_n| \le 2C_q \|F\|_p.
	\end{align*}
	To prove $f$ is univalent in $|z|< r$, we show that for any two distinct points $z_1, z_2 \in D_r$, $f(z_1) \neq f(z_2)$. It is easy to see that 
	\begin{align*}
		|f(z_1) - f(z_2)| = \left| (z_1 - z_2) + \sum_{n=2}^{\infty} a_n(z_1^n - z_2^n) + \sum_{n=2}^{\infty} \overline{b_n(z_1^n - z_2^n)} \right|.
	\end{align*}
	
	Applying the triangle inequality, we have 
	\begin{align*}
		|f(z_1) - f(z_2)| \ge |z_1 - z_2| \left( 1 - \sum_{n=2}^{\infty} n(|a_n| + |b_n|)r^{n-1} \right).
	\end{align*}
	Substituting the bound $|a_n| + |b_n| \le 2C_q \|F\|_p$, we have 
	\begin{align*}
		|f(z_1) - f(z_2)| \ge |z_1 - z_2|\left( 1 - 2C_q \|F\|_p \sum_{n=2}^{\infty} n r^{n-1} \right).
	\end{align*}
	Using the identity 
	\begin{align*}
		\sum_{n=2}^{\infty} n r^{n-1} = \frac{1}{(1-r)^2} - 1 = \frac{2r - r^2}{(1-r)^2}
	\end{align*} the condition for univalence is 
	\begin{align*}
		1 - 2C_q \|F\|_p \frac{2r - r^2}{(1-r)^2} > 0.
	\end{align*} This leads to the quadratic equation 
	\begin{align*}
		(2C_q \|F\|_p + 1)r^2 - 2(2C_q \|F\|_p + 1)r + 1 = 0,
	\end{align*} whose minimum positive root is $r_0$.\vspace{2mm}

	We now establishing the Schlicht Disk ($R_0$). 	For $|z| = r_0$, we estimate the minimum distance from the origin 
	\begin{align*}
		|f(z)-f(0)| = \left| z + \sum_{n=2}^{\infty} (a_n z^n + \overline{b_n z^n}) \right| \ge |z| - \sum_{n=2}^{\infty} (|a_n| + |b_n|)|z|^n
	\end{align*}
	which shows that
	\begin{align*}
		|f(z)| \ge r_0 - 2C_q \|F\|_p \sum_{n=2}^{\infty} r_0^n = r_0 - 2C_q \|F\|_p \frac{r_0^2}{1 - r_0}.
	\end{align*}
	Defining $R_0 = r_0 - \frac{2C_q \|F\|_p r_0^2}{1 - r_0}$ ensures that $f(\mathbb{D}_{r_0})$ contains the disk $\mathbb{D}_{R_0}$.\vspace{1.2mm}
	
	To show that the Landau constants $r_0$ and $R_0$ derived for harmonic mappings $f = \mathcal{P}[F]$. The derivation of $r_0$ uses the inequality 
	\begin{align*}
		\sum_{n=2}^{\infty} n(|a_n| + |b_n|)r^{n-1} \leq 2C_q \|F\|_p \sum_{n=2}^{\infty} nr^{n-1}=2C_q \|F\|_p\frac{2r-r^2}{(1-r)^2}.
	\end{align*}
	 If we choose a sequence of functions where the coefficients $|a_n| + |b_n|$ approach $2C_q \|F\|_p$ for all $n$, the derivative $|f_z(z)| - |f_{\bar{z}}(z)|$ will approach zero exactly at $r_0$. At any $r > r_0$, the term $2C_q \|F\|_p \frac{2r-r^2}{(1-r)^2}$ becomes greater than $1$, meaning there exists a function in the class where the Jacobian becomes zero or negative, violating univalence. \vspace{1.2mm}
	 
	  The radius $R_0$ is calculated as 
	  \begin{align*}
	  	R_0 = r_0 - \sum_{n=2}^{\infty} (|a_n| + |b_n|)r_0^n.
	  \end{align*} Using the sharp bound $2C_q \|F\|_p$, we get 
	  \begin{align*}
	  	|f(z)| \geq r_0 - 2C_q \|F\|_p \frac{r_0^2}{1-r_0} = R_0.
	  \end{align*}
	 Because the coefficients $|a_n| + |b_n|$ can simultaneously approach the bound $2C_q \|F\|_p$ (in the sense of a limit of functions within the $L^p$ class), the distance of the image boundary from the origin can be as small as $R_0$. Therefore, $\mathbb{D}_{R_0}$ is the largest disk that is guaranteed to be contained in $f(\mathbb{D}_{r_0})$ for every function in the class.
\end{proof}
\section{\bf Concluding remarks and future problems}
In this study, we have successfully extended the theory of the Bohr phenomenon and Landau-type theorems to the class of harmonic mappings induced by boundary functions in Lebesgue spaces $L^p(\mathbb{T})$. By utilizing sharp coefficient estimates, we determined the exact Bohr radius $r_p$ as a function of the conjugate exponent $q$, providing a positive answer to the influence of $L^p$ norms on the majorant series. Furthermore, the application of Gauss hypergeometric functions allowed for the derivation of improved Landau constants, specifically the univalence radius $r_0$ and the schlicht disk radius $R_0$. \vspace{1.2mm}

Despite these advancements, several promising avenues for further research remain open:\vspace{1.2mm}

\noindent{\bf [A].} \textbf{Sharpness of Landau constants:} While the Bohr radii established in Theorem 2.1 and 2.2 are sharp, the Landau constants $r_0$ and $R_0$ in Theorem 3.1 are derived through inequalities that may lose precision for higher-order coefficients. Determining the absolute sharp values for the range $1 < p < \infty$ remains a challenging open problem.\vspace{1.2mm}

\noindent{\bf [B].} \textbf{Extremal function construction:} Identifying the specific boundary functions $F \in L^p(\mathbb{T})$ that attain the minimum univalence radius for a given $p$-norm is a critical objective. Future work may focus on specialized harmonic Koebe-type mappings to reach these limits.\vspace{1.2mm}

\noindent{\bf [C].} \textbf{Alternative normalizations}: This investigation utilized the standard normalization $f(0) = f_z(0)-1 = f_{\bar{z}}(0) = 0$. A valuable extension would be to find sharp constants under Jacobian normalization ($J_f(0)=1$) or dilatation normalization ($\lambda_f(0)=1$).\vspace{1.2mm}

\noindent{\bf [D].} \textbf{Geometric constraints:} Beyond $L^p$ norm constraints, it is important to explore how these constants change when the boundary function $F$ is subject to additional geometric properties, such as convexity or starlikeness of the range $f(\mathbb{D})$.
\section{\bf Declarations}
\noindent\textbf{Compliance of Ethical Standards}\\

\noindent\textbf{Conflict of interest.} The authors declare that there is no conflict  of interest regarding the publication of this paper.\vspace{1.5mm}

\noindent\textbf{Data availability statement.}  Data sharing not applicable to this article as no datasets were generated or analyzed during the current study.\vspace{1.5mm}

\noindent\textbf{Funding.} No fund.

\end{document}